%% file: gaussruler.tex
\documentclass[a4paper,11pt]{amsart}

\makeatletter
\@namedef{subjclassname@2020}{%
  \textup{2020} Mathematics Subject Classification}
\makeatother

\usepackage{tboege-preprint}
\usepackage{tboege-bordermatrix}
\usepackage{tboege-ci}
\usepackage{tboege-symbols}

\usepackage{multicol}
\usepackage{tikz}

\newcommand{\I}[1]{(\SF{#1})}

\newcommand{\SF}[1]{\mathsf{#1}}
\newcommand{\BBP}[1]{\BB{P#1}}

\newcommand{\ETR}{\mathbf{ETR}}
\newcommand{\ETK}{\mathbf{ET\BB{K}}}
\newcommand{\GCI}{\mathbf{GCI}}
\newcommand{\NP}{\mathbf{NP}}

\renewcommand{\tilde}[1]{\widetilde{#1}}

\usepackage{listofitems}
\newcommand\IP[1]{%
  \readlist*\mylist{#1}%
  \langle\mylist[1], \mylist[2]\rangle%
}
\newcommand\EIP[1]{%
  \readlist*\mylist{#1}%
  \llangle\mylist[1], \mylist[2]\rrangle%
}
\newcounter{xcoord}%
\newcommand\PC[1]{%
  \setsepchar{:}
  \ignoreemptyitems
  \readlist*\mylist{#1}%
  \left[%
    \setcounter{xcoord}{1}%
    \foreachitem\coord\in\mylist{%
      \coord%
      \ifthenelse{\thexcoord < \mylistlen}{\mathbin{:}}{}%
      \stepcounter{xcoord}%
    }%
  \right]%
}

\newcommand{\cross}{\times}

\newcommand{\oo}{\SF{o\mkern-2mu o}} 

\newcommand{\GF}{\mathrm{GF}}

\DeclareMathOperator{\sdeg}{sdeg}
\DeclareMathOperator{\rk}{rk}
\DeclareMathOperator{\colspan}{colspan}

\newtheorem*{namedthm}{\namedthmname}
\newcounter{namedthm}
\makeatletter
\newenvironment{named}[1]
  {\def\namedthmname{#1}%
   \refstepcounter{namedthm}%
   \namedthm\def\@currentlabel{#1}}
  {\endnamedthm}
\makeatother

\newtheorem{question}[theorem]{Question}

\title{Incidence geometry in the projective plane via \\
almost-principal minors of symmetric matrices}
\author{Tobias Boege}
\address{Tobias Boege, OvGU Magdeburg, Germany}
\email{tobias.boege@ovgu.de}
\date{\today}

\subjclass[2020]{51A25, 62R01, 14P10, 05B20, 15A15}
\keywords{%
  projective plane,
  von~Staudt construction,
  symmetric matrix,
  gaussoid,
  conditional independence,
  implication problem,
  rationality,
  universality%
}

\begin{document}

\begin{abstract}
We present an encoding of a polynomial system into vanishing and
non-vanishing constraints on almost-principal minors of a symmetric,
principally regular matrix, such that the solvability of the system
over some field is equivalent to the satisfiability of the constraints
over that field.
This implies two complexity results about Gaussian conditional independence
structures. First, all real algebraic numbers are necessary to construct
inhabitants of non-empty Gaussian statistical models defined by
conditional independence and dependence constraints. This gives a negative
answer to a question of Petr Šimeček. Second, we prove that the implication
problem for Gaussian CI is polynomial-time equivalent to the existential
theory of~the~reals.
\end{abstract}

\maketitle

\section{Introduction}

Matroids were conceived by Whitney in the 1930s as combinatorial abstractions
of the common properties of independence relations in vector spaces and graphs.
Today, matroid theory is a broad and active field of research with an extensive
corpus of theorems and constructions. Among the jewels of this theory are the
\emph{universality theorems} of Sturmfels and Mnëv:
to every affine variety~$\C V$ over an algebraically closed field, there
exists a rank-3 matroid whose projective realization space over this
field is birationally isomorphic to~$\C V$~\cite[Theorem~4.30]{BokowskiSturmfels};
and to every basic primary semialgebraic set~$\C K$ over a real-closed field,
a rank-3 oriented matroid may be found whose realization space is stably
equivalent to~$\C K$~\cite[Theorem~8.6.6]{OrientedMatroids},~\cite{RichterGebert}.

Interpreting the term of ``universality'' more liberally, the idea is to
show that certain objects or properties associated to the combinatorial
structure of a matroid can be arbitrarily complicated, within the confines
of an obvious upper bound and modulo a notion of equivalence which blurs
the concrete object but not its complexity.
In this looser sense, the $\NP$-completeness of orientability of rank-3
matroids by Richter-Gebert~\cite{OrientabilityNP} may count as a universality
result. Notably, also one of the first contributions to matroid theory, in
a paper by MacLane~\cite{MacLane}, gives a universality result:
it shows how to construct for every finite algebraic extension $\BB K/\BB Q$
a rank-3 matroid which is realizable over a field $\BB L/\BB Q$ if and only
if there is an embedding of fields $\BB K \hookrightarrow \BB L$.

In one way or another, all of these theorems rest on the ability of simple
rank-3 matroids to capture, prescribe and forbid incidence relations between
points and lines in the projective plane. To be precise, consider a
$3 \times n$ matrix~$A$ realizing a simple rank-3 matroid. Its columns
contain the homogeneous coordinates of $n$ points in the projective plane.
The non-bases of the matroid are those triplets $pqr$ where the
$3 \times 3$ subdeterminant of $A$ with columns $p$, $q$ and $r$,
denoted by the bracket $[pqr]$, vanishes. This is equivalent for the
indexed points $pqr$ to be \emph{collinear} in the projective plane.
The~collinearity predicate on point triplets is a geometric primitive
which allows the design of arbitrary planar incidence relations.
For example, to construct the intersection point $r$ of two distinct
lines, given by two pairs of distinct points $(p,q)$ and $(p',q')$,
one requires that $[pqr] = 0$ and $[p'q'r] = 0$, i.e.\ $r$ is on the
line $p \vee q$ and also on $p' \vee q'$.
Oriented matroids refine this to a left-right-or-on relation between
points and lines, such as studied by Knuth~\cite{AxiomsHulls}.

\pagebreak 

Building on these constructions with a ``projective ruler'' one
can execute the von~Staudt constructions and gain access, through the
synthetic geometry language of matroids, to the algebraic structure
underlying the projective plane.
\Cref{fig:Configurations} shows two configurations of lines and
(intersection) points whose projective incidence relation (also taking
into account parallelity via intersection points at infinity) can only
be realized over fields $\BB K/\BB Q$ containing $\sqrt5$ or~$\sqrt2$,
respectively.
The von~Staudt technique is a systematic way to design incidence
relations for which constructibility requires solutions to arbitrary
polynomial systems with integer coefficients.

\begin{figure}
\vspace{-1.3em}
\scalebox{0.6}{%
\begin{tikzpicture}[y=0.80pt, x=0.80pt, yscale=-0.6, xscale=0.6, inner sep=0pt, outer sep=0pt]
\input{perles.tikz}
\end{tikzpicture}
}
\hspace{3em}
\scalebox{0.6}{%
\begin{tikzpicture}[y=0.80pt, x=0.80pt, yscale=-0.5, xscale=0.5, inner sep=0pt, outer sep=0pt]
\input{hartshorne-sqrt2.tikz}
\end{tikzpicture}
}
\caption{The Perles configuration on the left requires $\sqrt{5}$
to be realized over characteristic zero~\cite[Section 5.5.3]{Gruenbaum}.
The affine configuration on the right requires
$\sqrt{2}$~\cite[Example 14.4.2]{BabyHartshorne}.}
\label{fig:Configurations}
\end{figure}
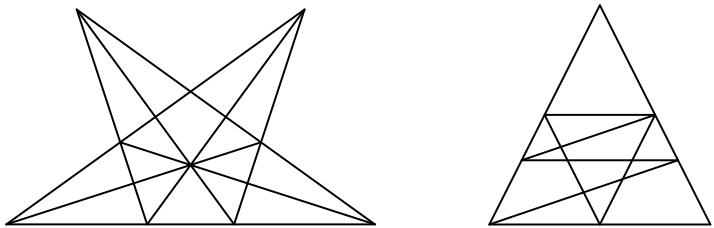

The theory of conditional independence (CI) structures follows the spirit
of matroid theory, in studying common properties of conditional independence
relations among random variables in probability theory, abstracting away the
probability distributions, like matroids abstract from vector spaces.
One branch of CI~structure theory is concerned with discrete random variables.
In~this context, the works of Fero Matúš are in particular to highlight, in
which a connection to matroid theory is drawn explicitly. In~\cite{MatusMatroids},
(simple) matroids are interpreted as a subset of special CI~structures
called semimatroids~\cite{MatusMatroids}, which are in turn semigraphoids.
This hierarchy of abstractions is the combinatorial shadow of the hierarchy of
representations: rank functions of hyperplane arrangements are rank functions
of subspace arrangements, which are in turn entropy functions of discrete
random variables, as far as the independence relations are concerned;
cf.~\cite[Section~6]{MatusMinors}.
Through this connection, essentially probabilistic questions can be raised
about matroid representations~\cite{MatusPartitions} and the related
rank inequalities are of interest in information theory and adjacent fields;
see the survey~\cite{InfoEqs}.
In the opposite direction, combinatorial or geometric ideas from matroid
theory have inspired developments in the theory of discrete CI~structures~\cite{MatusClassification,Studeny4CI}.

This article transfers the technique of von~Staudt constructions to encode
polynomial systems to the realm of \emph{Gaussian} instead of discrete
CI~structures.
Gaussian~CI studies the sets of vanishing almost-principal minors of
(positive-definite) symmetric matrices. The~similarity to matroid theory,
which studies sets of vanishing maximal minors of general matrices, is
apparent, but it is a different similarity from that which embeds linear
matroids into discrete CI~structures.
%
In \Cref{sec:Incidence} we show how to use almost-principal minor constraints,
similar to bases and non-bases in matroid theory, to encode plane projective
incidence relations which ultimately drive the von~Staudt encoding of a
polynomial system. Based on this technique, we obtain an analogue of the
universality result of MacLane in \Cref{sec:MacLane} and answer a question
posed by Petr Šimeček in~\cite{SimecekGaussian} negatively: there exist
non-empty Gaussian CI~models which contain no rational point. In fact, the
minimum algebraic degree of any point in such a model may be arbitrarily
high. Finally, in \Cref{sec:Implication}, we show that the implication problem
for Gaussian CI~statements is polynomial-time equivalent to the existential
theory of the reals. Possible extensions and future directions are discussed
in \Cref{sec:Remarks}.

\subsection*{Notation}

An introduction to the (real) projective plane including a thorough treatment
of the von~Staudt constructions, which are at the heart of this paper, can
be found in~\cite{Perspectives}. The~symbol $\IP{\cdot, \cdot}$ denotes the
standard scalar product of vectors and $\EIP{\cdot, \cdot}$ another symmetric
bilinear form whose definition will depend on context. We work with homogeneous
coordinates of points $p$ and lines $\ell$ in the projective plane $\BBP K^2$
over a field~$\BB K$. The set of points on the line~$\ell$ is
$\ell^\perp \defas \set{p \in \BBP K^2: \IP{p,\ell} = 0}$.

We adopt the convention that sans-serif symbols
$\SF i, \SF j, \SF k, \SF l, \SF p$ denote elements and $\SF E, \SF K,
\SF L, \SF P$ subsets of an implicit ground set $\SF N$ which indexes the
rows and columns of a symmetric matrix. Elements and singleton subsets
are not distinguished and juxtaposition abbreviates set union. For example
$\SF{iK}$ is short for $\set{\SF i} \cup \SF K \subseteq \SF N$ and $\SF{ij}$
denotes the two-element set $\set{\SF i, \SF j}$ and equals~$\SF{ji}$.
If $\Sigma$ is a symmetric matrix indexed by~$\SF N$, then $\Sigma_\SF{ij}$
addresses an entry of $\Sigma$ with a well-defined value.

\section{Preliminaries on Gaussian conditional independence}

Let $\xi$ denote a vector of random variables indexed by $\SF N$ with a joint distribution.
A conditional independence statement $\xi_\SF{i} \CI \xi_\SF{j} \mid \xi_\SF{K}$
asserts that the entries $\xi_\SF{i}$ and $\xi_\SF{j}$ are stochastically
independent under the distribution conditioned on the subvector~$\xi_\SF{K}$.
Informally, this means that any dependency between $\xi_\SF{i}$ and $\xi_\SF{j}$
is explained by the entries of $\xi_\SF{K}$. In the theory of conditional
independence one studies the sets of all symbols $\I{ij|K}$, denoting
CI~statements, which can simultaneously occur for a given class of distributions
or other objects for which ``$\I{ij|K}$'' can  be given an interpretation.
For~a proper introduction to the theory of (probabilistic) conditional
independence structures, the reader is referred to~\cite{Studeny}.

An $\SF N$-variate regular Gaussian distribution is given by its mean
vector~$\mu$ and its positive-definite covariance matrix~$\Sigma$.
Checking whether a CI~statement $\I{ij|K}$ holds is equivalent to evaluating
a special kind of subdeterminant of the covariance matrix, called an
\emph{almost-principal minor}:
\[
  \label{eq:CI} \tag{$\CI$}
  \text{$\I{ij|K}$ holds} \;\Leftrightarrow\;
  \apr{\SF{ij|K} : \Sigma} \defas \det \Sigma_{\SF{iK},\SF{jK}} = 0.
\]
This definition readily generalizes to symmetric matrices over arbitrary
fields, where positive-definiteness is replaced by \emph{principal regularity},
the condition that the principal minors $\pr{\SF K : \Sigma} \defas \det \Sigma_\SF{K}$
do not vanish. Principal regularity and its semidefinite extension,
positive-definiteness, may be seen as technical conditions which make
the interpretation of the CI~symbols $\I{ij|K}$ more well-behaved,
providing the theory with a notion of minors, duality, symmetries
and combinatorial constructions; cf.~\cite[Section~3]{Gaussant}.
The generalization of the structure of positive-definite matrices to~$\BB C$
was first undertaken by Matúš~\cite{MatusGaussian} when he derived the
axioms of \emph{gaussoids}~\cite{LnenickaMatus}.
While there exists a (combinatorially more involved) definition of
conditional independence for positive-semidefinite matrices, we work
with principally regular matrices in this paper and defer the treatment
of semidefinite matrices over $\BB R$ to \Cref{sec:Semidefinite}.

\begin{definition} \label{CIModel}
A set of \emph{CI~constraints} is a collection $\C I$ of conditional
independence $\I{ij|K}$ or dependence statements $\neg\I{ij|K}$ over a
ground set~$\SF N$. For any field $\BB K$, the \emph{model}~$\C V_{\BB K}(\C I)$
of~$\C I$ consists of all principally regular, symmetric $\SF N \times \SF N$
matrices~$\Sigma$ which \emph{satisfy}~$\C I$ as per~\eqref{eq:CI}.
For~ordered~fields~$\BB K$ the \emph{positive model}~$\C K_{\BB K}(\C I)$
is defined analogously as a subset of the positive-definite matrices.
\end{definition}
Principal regularity imposes $\pr{\SF K : \Sigma} \not= 0$ wheras
positive-definiteness imposes $\pr{\SF K : \Sigma} > 0$ on the model,
in addition to the CI~equations and inequations. Thus our models are
first-order definable sets in the theory of~$\BB K$. That is, if
$\BB K$ is any field $\C V_{\BB K}$ is a constructible subset of the
space of symmetric matrices, and if $\BB K$ is ordered, $\C K_{\BB K}$
is a semialgebraic subset, defined by polynomials with integer coefficients.
The language of CI~constraints allows to impose vanishing and non-vanishing
conditions on almost-principal minors of a matrix. The $\I{ij|K}$
almost-principal submatrix of a matrix $\Sigma$ contains the invertible
$\SF K \times \SF K$ principal submatrix. A~Schur complement expansion
of the determinant with respect to this block yields
\[
  \label{eq:CIIP} \tag{$\measuredangle$}
  \apr{\SF{ij|K} : \Sigma} = \pr{\SF K : \Sigma} \left(
    \Sigma_\SF{ij} -
    \Sigma_\SF{i,K} \, \Sigma_\SF{K}^{-1} \, \Sigma_\SF{K,j}
  \right) \overset{!}{=} 0.
\]
Thus prescribing a conditional independence $\I{ij|K}$ effectively stores
the scalar product of the row vector $\Sigma_\SF{i,K}$ and the column vector
$\Sigma_\SF{K,j}$ with respect to the Gram matrix $\Sigma_\SF{K}^{-1}$ in
the entry~$\Sigma_\SF{ij}$ of every matrix $\Sigma$ in the model of~$\I{ij|K}$.

\pagebreak 

\section{Polynomial systems as CI~constraints}
\label{sec:Incidence}

\subsection{Polynomial systems as ruler constructions}
\label{sec:PolynomialRuler}

The universality theorems for matroids rest on an encoding of arbitrary
polynomial systems in the bases and non-bases of a matroid. The relation
of the solution set of the polynomial system to the realization space of
the matroid depends on the technical finesse of this encoding and its
proof. For the results of this paper, we restrict ourselves to
\emph{solvability} questions: the task is to convert a system of polynomial
equations to a set of CI~constraints such that the constraints have a model
if and only if the system has a solution.
Polynomial systems, in turn, are modeled according to von~Staudt by certain
incidence relations in the projective plane. His~classical constructions
rely only on a basis of the projective plane and the ruler as a construction
tool. So~the encoding of certain plane projective ruler constructions is
our gateway to universality.

\begin{definition}
The~\emph{standard projective basis} consists of the infinite point on the
$x$-axis $\bm \infty_x = \PC{1:0:0}$, the infinite point on the $y$-axis
$\bm \infty_y = \PC{0:1:0}$, the origin $\bm 0 = \PC{0:0:1}$ and the
point of units $\bm 1 = \PC{1:1:1}$.
\end{definition}
From these points, the $x$- and $y$-axes $\ell_x$~and~$\ell_y$, unit points
on the axes $\bm1_x$~and~$\bm1_y$ and the line at infinity $\ell_\infty$ can be
constructed, which complete the framework in which ruler constructions are
carried~out. The~standard basis has favorable properties for the constructions
in the next section, notably its shape can be prescribed easily using CI~constraints,
which is why we insist~on~it.

\begin{definition}
A \emph{ruler construction} over a field~$\BB K$ is a finite list of
instructions which constructs a set of points and lines in $\BBP K^2$ from
a given set of points including the standard projective basis using the
computational primitives of
\begin{inparaenum}[label=(\alph*)]
\item \emph{joining} two already constructed, distinct points to form the line through them, and
\item \emph{meeting} two already constructed, distinct lines to form their intersection point.
\end{inparaenum}
The~construction algorithm may receive parameters in the form of
\emph{indeterminate points} which are placed on the $x$-axis $\ell_x$.
\end{definition}
Ruler constructions are required to be deterministic: by stipulating the
distinctness of joined points and met lines in~$\BBP K^2$, the resulting
line or point is uniquely defined as the one-dimensional space of solutions
to two independent linear equations in~$\BB K^3$. For instance, the line
$\ell$ through two distinct points $p, p'$ is given by $\IP{p,\ell} = 0$
and $\IP{p',\ell} = 0$. Usage of the indeterminate points and all objects
constructed from them is permitted as long as all joins and meets are
provably between distinct objects in every instantiation of the
indeterminates.
In this case, the join $\ell$ of the distinct points $p, p'$ can be
immediately computed by the cross product
\[
  \PC{p^x:p^y:p^z} \cross \PC{p'^x:p'^y:p'^z} \defas \PC{
    \det\begin{pmatrix} p^y & p'^y \\ p^z & p'^z \end{pmatrix} :
    -\det\begin{pmatrix} p^x & p'^x \\ p^z & p'^z \end{pmatrix} :
    \det\begin{pmatrix} p^x & p'^x \\ p^y & p'^y \end{pmatrix}
  }.
\]
The same operation computes, dually, the coordinates of the meet of
two distinct lines.


\begin{named}{Problem (F)} \label{ProblemF}
Given a system $F = \set{f_1, \dots, f_r}$ of integer polynomials
${f_i \in \BB Z[t_1, \dots, t_k]}$, construct with a ruler, starting
from the standard projective basis and an indeterminate point
$\bm t_i = \PC{t_i:0:1}$ for each unknown $t_i$, the points
$\bm f_i = \PC{f_i(t_1, \dots, t_k):0:1}$.
\end{named}

By introducing more equalities and variables, the polynomial system can
be assumed to contain only \emph{atomic} equations using the elementary
arithmetic operations of addition and multiplication: $t_i = t_j + t_k$
and $t_i = t_j \cdot t_k$, as well as a distinguished variable $t_1 = 1$
for the unit~\cite[Section~3.2]{vonStaudtSkew}. The von~Staudt constructions
implement precisely addition and multiplication of points on the $x$-axis.
Before we describe these algorithms, we construct a larger projective
framework out of the standard basis, containing points and lines which
are used~in~both:
\begin{multicols}{2}
\small
\textbf{Framework:}
\begin{enumerate}[series=Framework]
\item $\ell_x \defas \bm0 \cross \bm\infty_x = \PC{0:1:0}$
\item $\ell_y \defas \bm0 \cross \bm\infty_y = \PC{1:0:0}$
\item $\ell_\infty \defas \bm\infty_x \cross \bm\infty_y = \PC{0:0:1}$
\end{enumerate}

\begin{enumerate}[resume=Framework]
\item $\ell_{1x} \defas \bm1 \cross \bm\infty_x = \PC{0:-1:1}$
\item $\ell_{1y} \defas \bm1 \cross \bm\infty_y = \PC{-1:0:1}$
\item $\bm1_x \defas \ell_{1y} \cross \ell_x = \PC{1:0:1}$
\item $\bm1_y \defas \ell_{1x} \cross \ell_y = \PC{0:1:1}$
\end{enumerate}
\end{multicols}

\pagebreak 

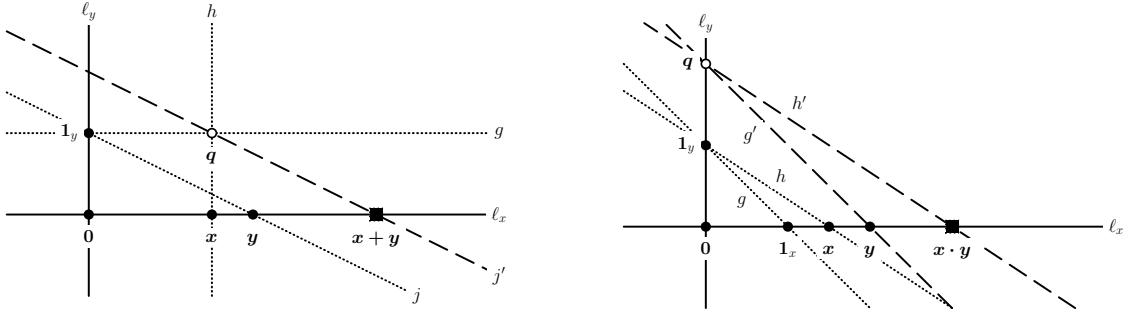
\begin{figure}
\scalebox{0.4}{%
\begin{tikzpicture}[y=0.80pt, x=0.80pt, yscale=-1, xscale=1, inner sep=0pt, outer sep=0pt]
\input{staudt-addition.tikz}
\end{tikzpicture}
}
\hspace{3em}
\scalebox{0.4}{%
\begin{tikzpicture}[y=0.80pt, x=0.80pt, yscale=-1, xscale=1, inner sep=0pt, outer sep=0pt]
\input{staudt-multiplication.tikz}
\end{tikzpicture}
}
\caption{Von~Staudt constructions in two affine pictures. The solid points
are given, the hollow ones are helper points in the construction of the
square target points. The axes are displayed as solid lines, helper lines
are dotted and the dashed lines, which are parallel to the dotted ones,
yield the target points.
}
\label{fig:vonStaudt}
\end{figure}

\Cref{fig:vonStaudt} contains pictures of the von~Staudt constructions
for  addition and multiplication of indeterminate points in the affine
$xy$-plane by projective ruler constructions from the standard~basis.
The pictures join points, meet lines and construct the parallel to a line
through another point. This last affine operation can be performed by the
projective ruler using the line at infinity not pictured here.
Full descriptions of these classical constructions are given in~\cite[Section~5.6]{Perspectives}
and with emphasis on matroids (over skew fields) in~\cite{vonStaudtSkew}.
Here we give the algorithms with indeterminates $\bm x = \PC{x:0:1}$ and
$\bm y = \PC{y:0:1}$ using cross products and using the same notation as
in \Cref{fig:vonStaudt}:

\begin{multicols}{2}
\small
\noindent\textbf{Addition:}
\begin{enumerate}
\item $g \defas \bm1_y \cross \bm\infty_x = \PC{0:-1:1}$
\item $h \defas \bm x \cross \bm\infty_y = \PC{-1:0:x}$
\item $\bm q \defas g \cross h = \PC{x:1:1}$
\item $j \defas \bm y \cross \bm1_y = \PC{-1:-y:y}$
\item $\bm\infty_j \defas j \cross \ell_\infty = \PC{-y:1:0}$
\item $j' \defas \bm q \cross \bm\infty_j = \PC{-1:-y:x+y}$
\item $\bm{x+y} \defas j' \cross \ell_x = \PC{x+y:0:1}$
\end{enumerate}
\hphantom{balancing line}

\noindent\textbf{Multiplication:}
\begin{enumerate}
\item $g \defas \bm1_x \cross \bm1_y = \PC{-1:-1:1}$
\item $h \defas \bm x \cross \bm1_y = \PC{-1:-x:x}$
\item $\bm\infty_g \defas g \cross \ell_\infty = \PC{-1:1:0}$
\item $\bm\infty_h \defas h \cross \ell_\infty = \PC{-x:1:0}$
\item $g' \defas \infty_g \cross \bm y = \PC{1:1:-y}$
\item $\bm q \defas g' \cross \ell_y = \PC{0:y:1}$
\item $h' \defas \bm q \cross \bm\infty_h = \PC{1:x:-x\cdot y}$
\item $\bm{x\cdot y} \defas h' \cross \ell_x = \PC{x\cdot y:0:1}$
\end{enumerate}
\end{multicols}

\begin{lemma} \label{SolveProblemF}
Given the standard basis, the von~Staudt constructions solve \ref{ProblemF}.
\qed
\end{lemma}

This very analytic treatment of the construction is required to observe
the following subtle point which will be important in \Cref{sec:RulerCI}:

\begin{lemma} \label{Preparation}
All meet and join operations in the von~Staudt construction are between
distinct points and lines, independently of the positions of the
indeterminates $\bm t_i$ on the $x$-axis.
Moreover, for every point and line needed in the construction, one
homogeneous coordinate can be given which is non-zero, also independently
of the indeterminates.
\qed
\end{lemma}

\subsection{Ruler constructions as CI~constraints}
\label{sec:RulerCI}

To model a ruler construction as CI~constraints, we work over a ground set
$\SF N = \SF{PLE}$, which decomposes into sets $\SF P = \set{\SF p_1, \SF p_2, \dots}$
and $\SF L = \set{\SF l_1, \SF l_2, \dots}$ for labeling the points and lines
which are used during the algorithm, respectively, and $\SF E = \set{\SF x, \SF y, \SF z}$
which indexes the homogeneous coordinates of the points and lines.
Instead of implementing the join and meet primitives via collinearity of
points, as matroids do (and where all lines are implicit), or by the cross
product, we use the scalar product interpretation of CI~constraints given
in~\eqref{eq:CIIP}. Namely, given homogeneous coordinates of a point $p$
and of a line $\ell$, the incidence~$p \in \ell^\perp$ is equivalent to the
vanishing of the standard scalar product~$\IP{p,\ell}$. To~construct the
line $\ell_{pq}$ joining two already constructed, distinct points $p, q$
labeled by $\SF p, \SF q \in \SF P$, introduce a new variable $\SF{l_{pq}}$
into the set $\SF L$ and require the incidence of the point~$\SF p$ and
the point~$\SF q$ to the new line~$\SF{l_{pq}}$.

The complete encoding of the von~Staudt constructions of a polynomial system
into a set of CI~constraints is given in \Cref{Constraints}~below. Up to
some implementation details, the basic ideas behind this \namecref{Constraints}~are:
\begin{itemize}
\item
The homogeneous coordinates of the point indexed by $\SF p \in \SF P$
are stored in the entries $\Sigma_\SF{pe}$ of a model~$\Sigma$, for
$\SF e \in \SF E = \set{\SF x, \SF y, \SF z}$. Likewise for lines
$\SF l \in \SF L$.

\item
For every pair of distinct points and/or lines $\SF a, \SF b \in \SF{PL}$,
we impose the CI~statement $\I{ab|xyz}$ in order to store the scalar product
of their homogeneous coordinates in the entry $\Sigma_\SF{ab}$. This scalar
product is with respect to the inverse block matrix $\Sigma_\SF{E}^{-1}$,
according to \eqref{eq:CIIP}.

\item
The desired orthogonalities between $\SF p \in \SF P$ and $\SF l \in \SF L$
which assert incidence relationships can then be prescribed with
CI~constraints~$\I{pl|}$.

\end{itemize}

\begin{definition} \label{Constraints}
Let $F = \set{f_1, \dots, f_r} \subseteq \BB Z[t_1, \dots, t_k]$.
Consider the von~Staudt construction of these polynomials making reference
to points labeled $\SF P = \set{\SF t_1, \dots, \SF t_k, \SF f_1, \dots, \SF f_r,
\SF p_1, \dots, \SF p_n}$ and lines labeled $\SF L = \set{\SF l_1, \dots, \SF l_m}$,
where the $\SF t_i$ represent the indeterminate points and $\SF f_i$
represent the values of the $f_i$ in the construction. Define a set of
CI~constraints~$\tilde{\C I}(F)$ over the ground~set~$\SF{PLE}$ consisting~of:
\begin{enumerate}[label=($\C I$.\roman*),series=Constraints]
\item \label{Constraints:C}
$\I{pe|}$ or $\neg\I{pe|}$ for all points $\SF p$ corresponding to the
standard projective basis and $\SF e \in \SF E$, depending on whether
the $\SF e$-coordinate of the point is zero or not.

\item \label{Constraints:T}
$\I{ty|}$ and $\neg\I{tz|}$ for indeterminate points~$\SF t = \SF t_1, \dots, \SF t_k$.

\item \label{Constraints:PL}
$\neg\I{ae|}$ for each $\SF a \in \SF{PL}$ and one of the coordinates
$\SF e \in \SF E$ on which the point or line labeled $\SF a$ is
non-zero, which can be deduced by \Cref{Preparation}.

\item \label{Constraints:IP}
$\I{ab|xyz}$ for all distinct $\SF a, \SF b \in \SF{PL}$.

\item \label{Constraints:I}
$\I{pl|}$ for any incidence relationship between $\SF p \in \SF P$ and
$\SF l \in \SF L$ which is required to express a join or meet operation
of the construction.

\end{enumerate}
\end{definition}

\Cref{fig:IncidenceMatrix} shows the generic matrix satisfying
constraint type~\ref{Constraints:IP}.
The constraints~$\tilde{\C I}(F)$ emulate the incidence relations behind
the von~Staudt construction. Every matrix which satisfies $\tilde{\C I}(F)$
gives values to the parameters $t_1, \dots, t_k$ and all other points and
lines such that the same incidence relations hold, which forces the $\SF f_i$
to assume the evaluation of $f_i(t_1, \dots, t_k)$ up to various scalings.
The caveat, however, is that each model starts the construction with
coordinates of points which are not necessarily the standard basis and
executes the ruler construction with a possibly non-standard notion of
``incidence'' which comes from~\eqref{eq:CIIP} and constraint
type~\ref{Constraints:IP}: the linear system defining incidence
$p \in \ell^\perp$ switches from $\IP{p,\ell} = 0$ to $\EIP{p,\ell} = 0$,
where $\EIP{\cdot,\cdot}$ is a non-degenerate symmetric bilinear form
defined by the inverse of~$\Sigma_\SF{E}$ in the matrix~$\Sigma$ thought
of as executing the ruler construction.

\begin{figure}
\begin{align*}
\kbordermatrix{
         &     \SF p_1     &    \cdots    &   \SF p_n    &        &     \SF l_1      &      \cdots     &      \SF l_m       &        &   \SF x  &   \SF y  &   \SF z  \\
 \SF p_1 &      p_1^*      &              &  \EIP{p,p'}  & \vrule &                  &                 &                    & \vrule &   p_1^x  &   p_1^y  &   p_1^z  \\
  \vdots &                 &    \ddots    &              & \vrule &                  &   \EIP{p,\ell}  &                    & \vrule &          &  \vdots  &          \\
 \SF p_n &    \EIP{p',p}   &              &    p_n^*     & \vrule &                  &                 &                    & \vrule &   p_n^x  &   p_n^y  &   p_n^z  \\ \cline{2-12}
 \SF l_1 &                 &              &              & \vrule &     \ell_1^*     &                 &  \EIP{\ell,\ell'}  & \vrule & \ell_1^x & \ell_1^y & \ell_1^z \\
  \vdots &                 & \EIP{\ell,p} &              & \vrule &                  &      \ddots     &                    & \vrule &          &  \vdots  &          \\
 \SF l_m &                 &              &              & \vrule & \EIP{\ell',\ell} &                 &      \ell_m^*      & \vrule & \ell_m^x & \ell_m^y & \ell_m^z \\ \cline{2-12}
 \SF x   &      p_1^x      &              &    p_n^x     & \vrule &     \ell_1^x     &                 &      \ell_m^x      & \vrule &          &          &          \\
 \SF y   &      p_1^y      &    \cdots    &    p_n^y     & \vrule &     \ell_1^y     &      \cdots     &      \ell_m^y      & \vrule &          & \Sigma_\SF{E} &     \\
 \SF z   &      p_1^z      &              &    p_n^z     & \vrule &     \ell_1^z     &                 &      \ell_m^z      & \vrule &          &          &
}
\end{align*}
\caption{The generic matrix satisfying the encoding of incidence relations
among points $\SF p_1, \dots, \SF p_n$ and lines $\SF l_1, \dots, \SF l_m$
in the projective plane, according to \Cref{Constraints}.
The scalar product $\llangle\cdot, \cdot\rrangle$ is given by the inverse
of the $\Sigma_\SF{E}$~block.}
\label{fig:IncidenceMatrix}
\end{figure}

\begin{lemma} \label{Model}
In the notation of \Cref{Constraints}, let $\Sigma \in \tilde{\C V}_{\BB K}(F)
\defas \C V_{\BB K}(\tilde{\C I}(F))$.
\begin{enumerate}[label=(\arabic*)]
\item \label{Model:Coords}
$\Sigma$ contains the homogeneous coordinates of points $\bm t_i$, $\bm f_i$,
$p_i$ and lines $\ell_i$ in $\BBP K^2$ in the entries $\SF P \times \SF E$
and $\SF L \times \SF E$. The image of the projective basis coincides with
the standard projective basis, except for $\tilde{\bm1} = \PC{s_x:s_y:1}$,
which may be different from $\bm1$. The $x$-axis, the $y$-axis and the line
at infinity are the same as with the standard projective basis.
The points $\bm t_i$ and $\bm f_i$ lie on the $x$-axis.

\item \label{Model:Incidence}
With the non-degenerate symmetric bilinear form $\EIP{\cdot, \cdot}$
defined by $\Sigma_{\SF E}^{-1}$, an incidence $\SF p_i \!\in\! \SF l_j^\perp$
imposed in the construction implies ${\EIP{p_i, \ell_j} = 0}$,
i.e.~${p_i \!\in\! \Sigma_\SF{E}(\ell_j^\perp)}$.

\item \label{Model:Determined}
The $\bm f_i$, $p_i$ and $\ell_i$ are uniquely determined as points in
$\BBP K^2$ by the points $\bm t_i$, the scalings $s_x$, $s_y$ and the
$\Sigma_\SF{E}$ block. All other off-diagonal entries of~$\Sigma$ are
functions of these homogeneous coordinates.
\end{enumerate}
\end{lemma}

\begin{proof}
\ref{Model:Coords}
By the relations~\ref{Constraints:C}, we have $\tilde{\bm\infty}_x =
\PC{1:0:0}$, $\tilde{\bm\infty}_y = \PC{0:1:0}$, $\tilde{\bm0} = \PC{0:0:1}$
and $\tilde{\bm1} = \PC{s_x:s_y:1}$ as points in the projective plane,
with $s_x, s_y \not= 0$. This is still a projective basis and the
$x$-axis, the $y$-axis and the line at infinity remain the same.
This is consistent with constraints~\ref{Constraints:T}, proving
that indeterminate points are on the $x$-axis. The $\bm f_i$ are
constructed by von~Staudt as intersection points with $\ell_x$,
so they remain on the $x$-axis. Because of constraints~\ref{Constraints:PL}
all homogeneous coordinate vectors are non-zero and hence valid
points/lines in~$\BBP K^2$.

\ref{Model:Incidence}
Denote by $\EIP{v,w} \defas v^\trans \, \Sigma_{\SF E}^{-1} \, w$ the
non-degenerate symmetric bilinear form defined by $\Sigma_{\SF E}^{-1} =
\Sigma_{\SF{xyz}}^{-1}$.
The relations $(\SF p_i \SF l_j|\SF{xyz})$ of type~\ref{Constraints:IP}
are equivalent to
\[
  \Sigma_{\SF p_i \SF l_j} \overset{!}{=}
    \Sigma_{\SF p_i,\SF{xyz}} \, \Sigma_{\SF{xyz}}^{-1} \, \Sigma_{\SF{xyz},\SF l_j}
    = \EIP{p_i, \ell_j}
\]
and then type~\ref{Constraints:I} makes this scalar product vanish, for
every relation $\SF p_i \in \SF l_j^\perp$~requested.

\ref{Model:Determined}
Since all points and lines are valid objects in $\BBP K^2$~--- in particular
due to type~\ref{Constraints:PL}, the zero vector is never permissible as
a vector of homogeneous coordinates (even though it satisfies all incidence
relations it may be involved in), thus the construction never degenerates~---,
the~uniqueness of the result of the von~Staudt construction in~\Cref{Preparation}
proves that all points and lines are uniquely determined by the starting points,
which are the projective basis and the indeterminates, as well as the definition
of incidence. Relations~\ref{Constraints:IP} then fix all off-diagonal entries
on $\SF{PL} \times \SF{PL}$ as functions of the homogeneous coordinates,
as per~\eqref{eq:CIIP}.
\end{proof}

\Cref{Model} shows that constraints~\ref{Constraints:C} for the standard
projective basis fix the \emph{standard} projective basis in every model
up to a scaling of the $x$- and $y$-axis by non-zero quantities $s_x$~and~$s_y$.
The points $\bm f_i$ which correspond to the evaluations of polynomials
$f_i$ end up on the $x$-axis and their location is uniquely determined by
the scalings, the bilinear form and the locations of $\bm t_i$.
The next \namecref{Evaluation} makes this more precise:

\begin{lemma} \label{Evaluation}
Let $\Sigma \in \tilde{\C V}_{\BB K}(F)$. Denote by $\bm t_i = \PC{t_i:0:1}$
and $\bm f_i = \PC{f_i^x:0:1}$ the points in $\BBP K^2$ determined by
$\Sigma$ according to \Cref{Model}.
Then $f_i^x = s_x f_i(t_1/s_x, \dots, t_k/s_x)$.
\end{lemma}

\begin{proof}
Let $S = \begin{psmallmatrix} s_x & 0 & 0 \\ 0 & s_y & 0 \\ 0 & 0 & 1
\end{psmallmatrix}$ be the scaling which maps the standard projective
basis to the one contained in~$\Sigma$.
Use $p'$ and $\ell'$ to refer to points constructed with the von~Staudt
algorithm and the standard basis and use $p$ and $\ell$ for the same
objects constructed by~$\Sigma$ with the scaled basis and non-standard
scalar product.
Consider also the points $\bm t_i' = \PC{t_i/s_x:0:1}$. Then we obtain
the projective basis and the indeterminates of~$\Sigma$ as images under~$S$
of the standard basis and the indeterminates~$\bm t_i'$.
It is straightforward to show by induction on the steps of the ruler
construction using \Cref{Model}~\ref{Model:Incidence} that:
\begin{itemize}
\item
If $\ell$ is constructed from $p_1$ and $p_2$ with
$p_i = S p_i'$, then $\ell = \Sigma_\SF{E} S^{-1} \ell'$.

\item
If $q$ is constructed in turn from $\ell_1$ and $\ell_2$ with
$\ell_i = \Sigma_\SF{E} S^{-1} \ell_i'$, then~$q = S q'$.
\end{itemize}
Thus, $\bm f_i = S \bm f_i'$ and $f_i^x = s_x f_i'^x = s_x f_i(t_1/s_x, \dots, t_k/s_x)$
by \Cref{SolveProblemF}.
\end{proof}

\pagebreak 

\begin{lemma} \label{ModelEval}
Let $a_1, \dots, a_k \in \BB K$ be arbitrary. If $\BB K$ is infinite,
then $\tilde{\C V}_{\BB K}(F)$ has a model which gives indeterminates
the value $\bm t_i = \PC{a_i:0:1}$ and evaluations $\bm f_i =
\PC{f_i(a_1, \dots, a_k):0:1}$.
Likewise for $\tilde{\C K}_{\BB K}(F)$ if $\BB K$ is ordered.
\end{lemma}

\begin{proof}
A model $\Sigma$ can be constructed based on \Cref{fig:IncidenceMatrix}.
Fill the coordinates of points of $\SF P$ corresponding to the standard
projective basis with the actual standard projective basis, set the
indeterminate points $\bm t_i$ as required and finally set $\Sigma_\SF{E}$
to be the identity matrix. Then for every incidence $\SF p \in \SF l^\perp$
demanded by the von~Staudt construction we have, by
\Cref{Model}~\ref{Model:Incidence} with $\EIP{\cdot,\cdot} = \IP{\cdot,\cdot}$
that indeed $\IP{p,\ell} = 0$, i.e.\ $p \in \ell^\perp$. Execute the
von~Staudt construction from these settings and all off-diagonal entries
will be filled to satisfy the~constraints.

This gives values to all entries of~$\Sigma$ except for the diagonals
$p_i^*$ and $\ell_j^*$ as shown in \Cref{fig:IncidenceMatrix}. These
diagonals are used to make $\Sigma$ principally regular. Consider any
principal submatrix of~$\Sigma$:
\[
  \qquad
  \Sigma_\SF{UV} = \kbordermatrix{
        & \SF U & \SF V \\
  \SF U &   A   &   B   \\
  \SF V & B^\trans & C
  }, \;
  \text{with $\SF U \subseteq \SF{PL}$ and $\SF V \subseteq \SF{E}$}.
\]
The $C$ block is a principal submatrix of $\Sigma_\SF{E}$, which is the
identity matrix, so $C$ is an identity as~well. Then it suffices to
show that the Schur complement $A - B B^\trans$ of~$C$ is regular.
The diagonals $p_i^*$ and $\ell_j^*$ appear only in $A$.
Similarly to \cite[Lemma~4.1]{Gaussant} one can show that a generic
choice of diagonal elements over an infinite field makes the determinant
of this Schur complement non-zero. If~$\BB K$~is ordered, the same
argumentation yields positivity --- $C$~is~positive-definite and one
may even choose the diagonals~of~$A$ so large that the Schur complement
becomes diagonally dominant and therefore positive-definite.
\end{proof}

\begin{definition} \label{ConstraintsZero}
Let $F = \set{f_1, \dots, f_r} \subseteq \BB Z[t_1, \dots, t_k]$.
Denote by $\C I(F)$ the extension of $\tilde{\C I}(F)$ by
\begin{enumerate}[label=($\C I$.\roman*),resume=Constraints]
\item \label{Constraints:F}
$\I{fx|}$ for all polynomial value symbols~$\SF f = \SF f_1, \dots, \SF f_r$
\end{enumerate}
and let $\C V_{\BB K}(F) \defas \C V_{\BB K}(\C I(F))$ and
$\C K_{\BB K}(F) \defas \C K_{\BB K}(\C I(F))$ over ordered fields.
\end{definition}

\begin{proposition} \label{vonStaudt}
Let $F = \set{f_1, \dots, f_r} \subseteq \BB Z[t_1, \dots, t_k]$.
The set $\C I(F)$ of CI~constraints has a (principally regular) model
over an infinite field $\BB K$ if and only the variety of $F$ (over
the algebraic closure $\ol{\BB K}$) has a $\BB K$-rational point.
The same is true for ordered fields and positive~models.
\end{proposition}

\begin{proof}
Suppose that there exists a model $\Sigma \in \C V_{\BB K}(F)$ which
contains points $\tilde{\bm1} = \PC{s_x:s_y:1}$ as well as $\bm t_i =
\PC{t_i:0:1}$ and $\bm f_i = \PC{f_i^x:0:1}$. The homogeneous coordinates
are unique up to a scalar from $\BB K$. Then by \Cref{Evaluation} and
constraint~\ref{Constraints:F}, we have $0 = f_i^x = s_x f_i(t_1/s_x,
\dots, t_k/s_x)$. Since $s_x \not= 0$ and $t_i/s_x \in \BB K$, these
define a solution to~$F$~in~$\BB K$.

Conversely, let $a_1, \dots, a_k \in \BB K$ be a solution to~$F$.
Then they define a model of $\tilde{\C V}(F)$ by \Cref{ModelEval}.
Moreover, since the $a_i$ are roots of the polynomials, the
constraints~\ref{Constraints:F} are satisfied by~\Cref{ModelEval}.
The proof works likewise for $\C K$.
\end{proof}

\section{Field extensions and a question of Šimeček}
\label{sec:MacLane}

In this section we apply \Cref{vonStaudt} to prove a universality result
for Gaussian CI~constraints and field extensions which can be used to answer
a question of Petr~Šimeček. We work in the Gaussian setting, which means
ordered fields below $\BB R$. The central notion is that of \emph{semialgebraic degree},
which measures the algebraic complexity of semidefinite Gaussian~models.

Satisfying a set of CI~constraints is equivalent to solving a system of
polynomial equations, inequations and inequalities with integer coefficients.
Tarski's transfer principle~\cite[Chapter~2.4]{ModelTheory} implies that
if a model exists over some ordered field~$\BB K$, then there exists a
model in the real algebraic numbers $\ol{\BB Q} \cap \BB R$. Since every
entry of such a model has finite algebraic degree over~$\BB Q$,
a model can already be found in a finite real extension of~$\BB Q$.
Therefore it is sensible to measure the algebraic complexity of a CI~model
in terms of field extension degrees. This~quantity is finite whenever the
constraints are satisfiable at all.

\begin{definition}
The \emph{semialgebraic degree} of~$\C I$ is the minimal extension degree
over~$\BB Q$ which is required to satisfy~$\C I$:
\[
  \sdeg \C I \defas
    \min_{\Sigma \in \C K_{\BB R}(\C I)}
    \max_{ij} \deg_{\BB Q} \Sigma_{ij}.
\]
\end{definition}

The reference to $\BB R$ in this definition is natural in the context of
statistics, but the notion remains the same if $\BB R$ is replaced by any
real-closed field such as $\ol{\BB Q} \cap \BB R$. Notice that both,
conditional independence and dependence statements, are necessary to
make this notion interesting: without dependence statements, the identity
matrix satisfies the constraints, whereas without independence statements,
any generic rational positive-definite matrix does.

Petr Šimeček in~\cite{SimecekGaussian} asked if every non-empty Gaussian
CI~model has a rational point. For the special case of regular Gaussians,
we can rephrase this question as follows:

\begin{named}{Šimeček's Question} \label{SimecekQ}
If the semialgebraic degree of a model is finite, then is it always~one?
\end{named}

Using the von~Staudt constructions of \Cref{vonStaudt}, we are able to
recreate the proof of MacLane~\cite{MacLane}, which strongly implies a
negative answer to this question.

\begin{theorem} \label{MacLane}
Let $\BB K / \BB Q$ be a real, finite field extension. There exists a set~$\C I$
of CI~constraints such that $\C K_{\BB L}(\C I)$ is non-empty for some extension
$\BB L / \BB Q$ if and only if there exists an embedding of fields~$\BB K
\hookrightarrow \BB L$.
\end{theorem}

\begin{proof}
The prime field $\BB Q$ is perfect and hence by the primitive element
theorem the finite extension~$\BB K$ has a primitive element $\alpha$
over $\BB Q$ with minimal polynomial~${f \in \BB Z[t]}$.
Application of \Cref{vonStaudt} to $F = \set{f}$ produces a
constraint~set~$\C I(f)$ which has a positive model over $\BB L$ if
and only if $\BB L$ contains a root of~$f$. By~standard facts about
field extensions, this implies that $\BB K$ is contained in any
field~$\BB L$ which has a model of~$\C I(f)$.
\end{proof}

\begin{corollary}
All real algebraic numbers are necessary to witness the non-emptiness
of regular Gaussian CI~models.
\qed
\end{corollary}

\section{Hardness of the implication problem}
\label{sec:Implication}

A set of CI~constraints $\C I$ decomposes into two sets of CI~statements
$\C L$ and $\C M$ such that $\C I = \C L \cup \neg \C M$, where the
negation is applied to every element of~$\C M$. The model~$\C K_{\BB R}(\C I)$
is the set of counterexamples to the validity of the \emph{conditional
independence inference formula}
\[
  \varphi(\C I): \bigwedge \C L \;\Rightarrow\; \bigvee \C M,
\]
which expresses that every regular Gaussian distribution which satisfies
all independencies in~$\C L$ must satisfy at least one of the independencies
in~$\C M$.
Denote the decision problem about the validity (with respect to regular
Gaussians) of an input inference formula by $\GCI$. Since the set of
counterexamples $\C K_{\BB R}(\C I)$ is semialgebraic, $\GCI$ reduces
to $\ETR$, the decision problem for the existential theory of the
reals; see~\cite[Chapter~13]{AlgorithmsReal} for a definition and
algorithms.

In the previous sections we have cast incidence geometry in the plane
into Gaussian CI~constraints. Reasoning about incidence statements in
the plane is hard because the valid incidence theorems are the axioms
of linear rank-3 matroids, and so we expect that reasoning about Gaussian
conditional independence inference is hard as well. This is the main
result of this~section:

\begin{theorem} \label{ETR}
There exists a polynomial-time algorithm to compute, for a basic
semialgebraic set~$\C K$ defined by polynomial constraints with
integer coefficients, two sets $\C L$~and~$\C M$ of CI~statements
such that $\bigwedge \C L \;\Rightarrow\; \bigvee \C M$ is valid for
regular Gaussians if and only if~$\C K$ is empty.
\end{theorem}

We have to explain the input encoding of a system of polynomials to
make sense of the ``polynomial-time algorithm'' assertion.
The size of a system of polynomials is the number of symbols used
to write it down in the rudimentary language of ordered fields,
i.e.\ using the constants $0$~and~$1$, the operations ${+}, {-}, {\cdot}$,
variables and the relations ${=}, {\not=}, {<}, {\le}, {\ge}, {>}$.
One should note that by increasing the number of variables and
equations, it is possible to assign powers of two to certain
variables and then encode coefficients and exponents efficiently
with binary coding~length, if desired. This encoding of a polynomial
system is the same that was applied in \Cref{sec:PolynomialRuler}
to model a system of equations via the von~Staudt constructions
for addition and~multiplication.

\begin{proof}[Proof of \Cref{ETR}]
Let any system of polynomial constraints $f_i \bowtie 0$ be given,
where $f_i \in \BB Z[t_1, \dots, t_k]$ and ${\bowtie} \in
\set{ {=}, {\not=}, {<}, {\le}, {\ge}, {>} }$. A standard construction
transforms this system into an equisatisfiable system of equations by
introducing for each non-equality constraint a new variable~$y$~and:
\begin{description}
\item[${\not=}$] replace $f \not= 0$ by $yf - 1 = 0$.
\item[${>}$] replace $f > 0$ by $y^2 f - 1 = 0$.
\item[${\ge}$] replace $f \ge 0$ by $f - y^2 = 0$.
\end{description}
This procedure takes polynomial time and produces an equivalent system
over Euclidean fields, since we replace a positivity constraint by an
equality to a square. In particular real-closed fields like $\BB R$
are Euclidean. The resulting equations are encoded into CI~constraints
via \Cref{vonStaudt} noting that the number of points constructed and
constraints~$\C I$ emitted by the algorithm is polynomial in the number
of basic arithmetic operations. Hence $\C I$ has polynomial size in the
coding length of the original system. As explained at the beginning of
this section, the points satisfying the polynomial system are precisely
the counterexamples to the validty of~$\varphi(\C I)$.
\end{proof}

\begin{remark}
In the opposite direction, $\GCI$ can be reduced to $\ETR$ in polynomial
time as~well by writing out the CI~constraints as a polynomial system for
a generic positive-definite matrix.
Care has to be taken to not allow the polynomials in the system to grow
too large. For~example, writing out an almost-principal minor as the
determinant of a generic $\SF{iK} \times \SF{jK}$ matrix using the
Leibniz formula takes at least $(|\SF K| + 1)!$ steps which is not
polynomial in the coding length of a CI~statement $\I{ij|K}$, which we
may suppose to be $|\SF{ijK}|$.

A polynomial-time procedure to write $\det \Sigma$ as a system of equations
employs Gaussian elimination simultaneously on rows and columns, producing
an $LDU$ decomposition of~$\Sigma$, so~that its determinant is the product
of the final diagonal elements. After the elimination of a row and column,
new variables for the remaining, modified entries of the matrix must be
introduced to avoid accumulating long polynomials.
\end{remark}

\section{Remarks}
\label{sec:Remarks}

The proofs of universality results for regular Gaussian~CI in \Cref{MacLane}
and \Cref{ETR} can be further generalized, as the tools developed in
\Cref{sec:Incidence} are largely independent of the field. The existence
proof of principally regular or positive-definite models in \Cref{ModelEval}
requires infinite fields. In the following we address some directions for
generalization and partial~results.

\subsection{Semidefinite models}
\label{sec:Semidefinite}

Positive-semidefinite matrices are the covariance matrices of singular
Gaussian distributions, whose probability mass is concentrated on a
proper subspace of $\BB R^\SF{N}$. This class of distributions has a
more involved definition of conditional independence~\cite{SimecekIndependence}:
\[
  \label{eq:CIsemi} \tag{${\CI\!}^-$}
  \text{$\I{ij|K}$ holds} \;\Leftrightarrow\; \apr{\SF{ij|L} : \Sigma} = 0,
\]
for any subset $\SF L \subseteq \SF K$ with $\pr{\SF L : \Sigma} \not= 0$
and $\rk \Sigma_\SF{L} = \rk \Sigma_\SF{K}$. The vanishing of this
almost-principal minor is independent of the choice of $\SF L$;
see~\cite[Section~A.8.3]{Studeny} and its references. This definition
generalizes~\eqref{eq:CI} for regular Gaussians. The \emph{model}
$\C K_{\BB K}^-(F)$ of a set of CI~constraints are then all
positive-semidefinite matrices which satisfy them. This model is still
semialgebraic, because for every assignment of ranks to the principal
minors of a generic symmetric matrix, there is a clear interpretation
of every CI~statement as an algebraic equation. Since rank conditions
are algebraic as well, the model is described by a union of basic
semialgebraic sets, one for each rank assignment.
However, the combinatorics of vanishing principal and almost-principal
minors and, depending on it, the interpretation of CI are much more
involved than in the regular Gaussian or principally regular case.

\pagebreak 

The CI~constraints of \Cref{Constraints} and \Cref{ConstraintsZero}
are special in this regard as the only ones which do not directly
address \emph{entries} of the matrix are of type~\ref{Constraints:IP}.
Therefore the only principal submatrix whose singularity would change
the interpretation of the constraints is $\Sigma_\SF{E}$. Indeed,
if this matrix was singular, the proofs in \Cref{sec:RulerCI}
would break down because joins and meets would no longer be unique.
As an undeserved extra of the encoding of the standard projective
basis and the following harmless additional relations, this never
happens:

\begin{enumerate}[label=($\C I$.\roman*),resume=Constraints]
\item \label{Constraints:S}
$\I{pq|}$ or $\neg\I{pq|}$ for all points $\SF p, \SF q$ of the
standard projective basis depending on whether $\IP{p,q} = 0$
or not, respectively.
\end{enumerate}
These constraints for pairs of points from $\set{\bm\infty_x,
\bm\infty_y, \bm0}$ imply that $\Sigma_\SF{E}$ is diagonal, which
does not impair the existence proofs in \Cref{ModelEval}.
The remaining combinations with $\bm1$ ensure that $\Sigma_\SF{E}$
is invertible:

\begin{lemma}
Let $F$ be a polynomial system and $\BB K$ an ordered field.
If $\Sigma \in \C K_{\BB K}^-(F)$ satisfies the additional
constraints~\ref{Constraints:S}, then $\Sigma_\SF{E}$ is
invertible.
\end{lemma}

\begin{proof}
Consider the generic matrix on $\set{\SF \oo_x, \SF \oo_y, \SF 0, \SF 1,
\SF x, \SF y, \SF z}$ corresponding to the projective basis and the
homogeneous coordinates which satisfy the constraints~\ref{Constraints:C},
\ref{Constraints:IP} and \ref{Constraints:S}:
\begingroup\small
\[
  \kbordermatrix{
              & \SF \oo_x & \SF \oo_y & \SF 0 & \SF 1 &        & \SF x & \SF y & \SF z \\
    \SF \oo_x &     *     &     0     &   0   &   *   & \vrule &   *   &   0   &   0   \\
    \SF \oo_y &     0     &     *     &   0   &   *   & \vrule &   0   &   *   &   0   \\
    \SF 0     &     0     &     0     &   *   &   *   & \vrule &   0   &   0   &   *   \\
    \SF 1     &     *     &     *     &   *   &   *   & \vrule &   *   &   *   &   *   \\ \cline{2-9}
    \SF x     &     *     &     0     &   0   &   *   & \vrule &       &       &       \\
    \SF y     &     0     &     *     &   0   &   *   & \vrule &   & \Sigma_\SF{E} &   \\
    \SF z     &     0     &     0     &   *   &   *   & \vrule &       &       &
  }
\]
\endgroup
where all occuring $*$ symbols denote non-zero numbers.
The submatrix $\Sigma_\SF{E}$ may have rank $0$, $1$, $2$ or $3$.
If the rank is zero, then $\I{ab|xyz}$ is equivalent to $\I{ab|}$
by~\eqref{eq:CIsemi}, but this contradicts $\neg\I{01|} \wedge
\I{01|xyz}$. If the rank is one, we may assume by the symmetry
of the projective basis that $\Sigma_\SF{x}$ is an invertible principal
submatrix of full rank in $\Sigma_\SF{E}$. Then $\I{ab|xyz}$ is
equivalent to~$\I{ab|x}$. But $\I{01|xyz} \Leftrightarrow \I{01|x}$
means
\begingroup\small
\[
  \apr{\SF{01|x} : \Sigma} = \det \kbordermatrix{
          & \SF 1 & \SF x \\
    \SF 0 &   *   &   0   \\
    \SF x &   *   & \Sigma_\SF{x}
  } \overset{!}{=} 0,
\]
\endgroup
which is a contradiction to $\neg\I{01|}$ and $\Sigma_\SF{x}$ having
rank~one. Lastly, if the rank is two and $\Sigma_\SF{xy}$ is, without
loss of generality, an invertible principal submatrix of full rank,
then the same contradiction arises from $\I{01|xyz} \Leftrightarrow
\I{01|xy}$ and $\neg\I{01|}$, as the almost-principal minor
\begingroup\small
\[
  \apr{\SF{01|xy} : \Sigma} = \det \kbordermatrix{
          & \SF 1 &     \SF x      &     \SF y      \\
    \SF 0 &   *   &       0        &       0        \\
    \SF x &   *   & \Sigma_\SF{x}  & \Sigma_\SF{xy} \\
    \SF y &   *   & \Sigma_\SF{yx} & \Sigma_\SF{y}
  } \overset{!}{=} 0
\]
\endgroup
again factors into the non-zero $\pr{\SF{xy} : \Sigma}$ and the non-zero
$\SF{01}$-entry. Hence $\Sigma_\SF{E}$ has full rank.
\end{proof}

It follows that every model of the extended constraints interprets the
CI~statements exactly like the principally regular models and performs
the incidence geometry with a non-degenerate symmetric bilinear form.
We recover \Cref{Evaluation} and together with \Cref{ModelEval}, which
holds a~fortiori, this proves \Cref{vonStaudt} for the semidefinite model.
This answers \ref{SimecekQ} also in the semidefinite setting:

\begin{corollary} \label{SemidefiniteMacLane}
All real algebraic numbers are necessary to witness the non-emptiness
of (semidefinite) Gaussian CI~models.
\qed
\end{corollary}

\begin{example}[Šimeček's model №~85]
The paper~\cite{SimecekGaussian} in which Petr Šimeček posed his
rationality question is concerned with the CI~structures realizable
by (not necessarily regular) Gaussian distributions on four variables.
For all but one of the models, №~85, he identified a rational
covariance matrix realizing it.

The CI~structure №~85 is
$\set{\I{12|4}, \I{14|3}, \I{14|23}, \I{24|3}, \I{24|13}, \I{34|12}}$.
The above \namecref{SemidefiniteMacLane} shows that there are semidefinite
CI~models without rational points, but №~85 is not one of them. The matrix
\begingroup\small
\[
  \kbordermatrix{
        &  \SF 1  &  \SF 2  &  \SF 3 &  \SF 4  \\
  \SF 1 &    1    &  -1/17  & -49/51 &  -7/17  \\
  \SF 2 &  -1/17  &    1    &   1/3  &   1/7   \\
  \SF 3 &  -49/51 &   1/3   &    1   &   3/7   \\
  \SF 4 &  -7/17  &   1/7   &   3/7  &    1
  }
\]
\endgroup
is positive-semidefinite and its vanishing principal minors are
$\pr{\SF{123} : \Sigma}$ and $\pr{\SF{1234} : \Sigma}$, which do not
affect the interpretation of any CI~statement via~\eqref{eq:CIsemi},
and the CI~structure realized by this matrix is №~85. The matrix was
found (quickly) by Wolfram~Mathematica~\cite{Mathematica} after
optimistically imposing the mentioned rank constraints and simplifying
the resulting polynomial system by hand.
\end{example}

\subsection{Finite fields}

\Cref{vonStaudt} does not apply to finite fields because principal regularity
is a restrictive condition for very small fields. For instance, the identity
matrix is the only principally regular matrix over $\GF(2)$. Principal regularity
is part of the definition of CI~model because it ensures a well-defined
interpretation of the CI~statement~$\I{ij|K}$, but afterwards this property
is almost trivially enforced in \Cref{ModelEval} and plays no role for
encoding of point and line configurations. For positive-semidefinite matrices,
there exists the more intricate definition~\eqref{eq:CIsemi}, but this does
not generalize to arbitrary symmetric matrices:

\begin{example} \label{CIsemiIlldefined}
Consider the symmetric (not positive-semidefinite) matrix
\begingroup\small
\[
  \Sigma =
  \kbordermatrix{
        & \SF i & \SF j & \SF x & \SF y & \SF z \\
  \SF i &   1   &   1   &  -1   &   0   &   0 \\
  \SF j &   1   &  -1   &  -1   &   0   &   0 \\
  \SF x &  -1   &  -1   &   1   &   0   &   1 \\
  \SF y &   0   &   0   &   0   &   2   &   0 \\
  \SF z &   0   &   0   &   1   &   0   &   1
  }.
\]
\endgroup
The interpretation of the CI~symbol~$\I{ij|xyz}$ according to~\eqref{eq:CIsemi}
is inconsistent. It depends on the choice of full-rank subset of~$\SF{xyz}$:
\begingroup\small
\begin{gather*}
  \pr{\SF{xyz} : \Sigma} = 0, \quad
  \pr{\SF{xy} : \Sigma} = 2 \not= 0, \quad
  \pr{\SF{yz} : \Sigma} = 2 \not= 0, \quad
  \pr{\SF{xz} : \Sigma} = 0, \\
  \apr{\SF{ij|xy} : \Sigma} = 2 - 2 = 0, \quad
  \apr{\SF{ij|yz} : \Sigma} = 2 \not= 0.
\end{gather*}
\endgroup
\end{example}

The proof of the well-definedness of~\eqref{eq:CIsemi} for semidefinite
matrices needs the well-definedness of the generalized Schur
complement~\cite[Section~A.8.1]{Studeny} which in turn rests
on the fact that for every block decomposition of a semidefinite matrix
$\Sigma = \begin{psmallmatrix} A & B \\ B^\trans & D \end{psmallmatrix}$
the $\colspan(B)$ is contained in $\colspan(A)$, i.e.\ there exists $B'$
such that $B = AB'$. This is shown in~\cite[Theorem~1.19]{SchurComplement}.
This $\colspan$ property on every block decomposition, for which
$\Sigma$ which may be called \emph{diagonally spanning}, is entirely
linear-algebraic and it can be defined over every field, where it ensures
the well-definedness of $\I{ij|K}$ via~\eqref{eq:CIsemi}. The matrix
in \Cref{CIsemiIlldefined} is symmetric but not diagonally spanning
as witnessed by $\begin{psmallmatrix} 0 \\ 0 \\ 1 \end{psmallmatrix} \not\in
\colspan \begin{psmallmatrix} 1 & 1 & -1 \\ 1 & -1 & -1 \\ -1 & -1 & 1
\end{psmallmatrix}$.
It is not obvious that the analogue of \Cref{ModelEval} holds for models
of diagonally spanning matrices over finite fields.

\begin{question}
Is it possible to define an interpretation of $\I{ij|K}$ for general
symmetric matrices which is consistent with semidefinite matrices?
Can the proof of \Cref{vonStaudt} be generalized?
\end{question}

Another attempt at generalization would explicitly involve vanishing
and non-vanishing statements about \emph{principal} minors in the
constraint~set. For results and open questions about the polynomial
relations among principal and almost-principal minors of a generic
symmetric matrix, see~\cite{Geometry}.

\subsection{Algebraic degrees}

The \emph{algebraic degree} of a constraint set~$\C I$ with respect
to some characteristic~$k$ (zero or a prime number)~is
$
  \deg_k \C I := \min_{\Sigma \in \C V_{\ol{\Bbbk}}(\C I)} \max_{ij} \deg_{\Bbbk} \Sigma_{ij},
$
where $\Bbbk$ is the prime field of characteristic~$k$ and $\ol{\Bbbk}$
its algebraic closure. The Lefschetz principle states that the emptiness
of a constructible set defined by integer polynomials over algebraically
closed fields does not depend on the particular field but only on the
characteristic; for a proof see~\cite[Chapter~2.4]{ModelTheory}.
Analogous to Tarski's transfer principle for real-closed fields, it
ensures the finiteness of this degree whenever a model over \emph{some}
field of the given characteristic exists.

Our techniques apply only to infinite prime fields, i.e.\ the rational
numbers, but they suffice to settle the algebraic analogue of \ref{SimecekQ}
posed in~\cite[Question~6.7]{Gaussant}:

\begin{corollary}
All algebraic numbers are necessary to witness the non-emptiness of
principally regular CI~models over $\BB Q$. In particular, for every
step in the chain $\BB Q \subseteq \BB R \subseteq \BB C$, there exist
CI~structures which are realizable over the next but not over the
previous~field.
\qed
\end{corollary}

\subsection{Existential theories}

Denote by $\ETK$ the problem to decide whether a constructible set defined
by integer polynomials over~$\BB K$ is empty or not and by $\ETK_{\le}$
the corresponding problem for semialgebraic sets over an ordered field.
We have analogous versions of the implication problem $\GCI(\BB K)$ and
$\GCI(\BB K_{\le})$.

\begin{corollary}
\begin{inparaenum}[label=(\arabic*)]
\item
$\ETK$ reduces to $\GCI(\BB K)$ when $\BB K$ is infinite.

\item
$\ETK_{\le}$ reduces to $\GCI(\BB K_{\le})$ whenever $\BB K$ is a
Euclidean ordered field.

\item
$\ETK$ and $\ETK_{\le}$ and hence $\GCI(\BB K)$ and $\GCI(\BB K_{\le})$
are decidable over algebraically and real-closed fields, respectively.
\end{inparaenum}
\qed
\end{corollary}

This implies that the classification problem attached to Šimeček's question,
i.e.\ to decide $\GCI(\BB Q)$, is famously open, as it is equivalent to
Hilbert's 10th problem over~$\BB Q$. For a survey of this problem,
see~\cite{DefiningIntegers}.

It is remarkable that the upper bound of $3$ can be placed on the size
of conditioning sets $\SF K$ in every CI~symbol $\I{ij|K}$ required in
the reduction of $\ETR$ to~$\GCI$ in \Cref{ETR}. On~the other hand, the
construction requires an unbounded number of each, antecedents~$\C L$
and consequents~$\C M$: antecedents to enforce incidence
relations~\ref{Constraints:IP}~and~\ref{Constraints:I} and consequents
to ensure that every point defined by the von~Staudt construction is a
valid point in projective space with constraint type~\ref{Constraints:PL}.
Prior research into infinite families of Gaussian inference rules has
usually targeted single-consequent formulas~\cite{SullivantNonfinite,SimecekNonfinite}.

\begin{question}
Is there a polynomial-time reduction of $\ETR$ to $\GCI$ for which there
is a universal upper bound on the number of~consequents in the constructed
inference formulas?
\end{question}

\subsection{Future work}

Given the encoding of projective plane incidence relations over fields
presented here, a natural direction for future work concerns skew~fields
with the same ambitions as in~\cite{vonStaudtSkew}. The more fine-grained
universality theorems mentioned in the introduction are an attractive
target as well. These concern topological properties of the realization
spaces of matroids. The analogue of matroids in this context would be
gaussoids~\cite{Geometry}, which are ``complete'' sets of CI~constraints
in the sense that each CI~statement appears in them either negated or
non-negated.
The extension to oriented gaussoids allows to impose sign constraints
on conditional dependencies. Statisically, this prescribes the signs
of correlations among random variables, and in our model of the plane
it would make the orientation of points with respect to lines expressible.

Recent work on discrete random variables~\cite{ArrangementsCI} considers
related applications of point and line configurations to the decomposition
of conditional independence ideals, in particular with hidden variables.
At~the Algebraic Statistics Online Seminar where this work was presented,
Thomas Kahle suggested that a universality theorem like \Cref{MacLane}
for discrete CI would settle a question of Matúš, which is entirely
parallel to and older than that of Šimeček:

\begin{question}[{\cite{MatusFinal}}]
Does every CI~structure which is realizable by discrete random variables
have a realization where every atomic probability is rational?
\end{question}

\subsection*{Acknowledgement}
The author would like to thank Thomas Kahle for his comments.
This~work is funded by the Deutsche Forschungsgemeinschaft
(DFG, German Research Foundation) -- 314838170, GRK 2297 MathCoRe.

\bibliographystyle{alpha}
\bibliography{gaussruler}

\end{document}

%% file: perles.tikz

\begin{scope}[rotate around={-90.0:(286.0,286.0)}]
  \path[cm={{1.0,0.0,0.0,1.0,(-66.0,-5.0)}},draw=black,line join=miter,line cap=round,miter limit=3.25,line width=1.240pt] (275.7070,53.0781) -- (275.7070,531.9219);
  \path[cm={{1.0,0.0,0.0,1.0,(-66.0,-5.0)}},draw=black,line join=miter,line cap=round,miter limit=3.25,line width=1.240pt] (275.7070,235.9805) -- (557.1641,144.5312);
  \path[cm={{1.0,0.0,0.0,1.0,(-66.0,-5.0)}},draw=black,line join=miter,line cap=round,miter limit=3.25,line width=1.240pt] (275.7070,53.0781) -- (557.1641,440.4688);
  \path[cm={{1.0,0.0,0.0,1.0,(-66.0,-5.0)}},draw=black,line join=miter,line cap=round,miter limit=3.25,line width=1.240pt] (275.7070,531.9219) -- (557.1641,144.5312);
  \path[cm={{1.0,0.0,0.0,1.0,(-66.0,-5.0)}},draw=black,line join=miter,line cap=round,miter limit=3.25,line width=1.240pt] (275.7070,349.0195) -- (557.1641,440.4688);
  \path[cm={{1.0,0.0,0.0,1.0,(-66.0,-5.0)}},draw=black,line join=miter,line cap=round,miter limit=3.25,line width=1.240pt] (275.7070,349.0195) -- (557.1641,144.5312);
  \path[cm={{1.0,0.0,0.0,1.0,(-66.0,-5.0)}},draw=black,line join=miter,line cap=round,miter limit=3.25,line width=1.240pt] (383.2148,383.9492) -- (275.7070,53.0781);
  \path[cm={{1.0,0.0,0.0,1.0,(-66.0,-5.0)}},draw=black,line join=miter,line cap=round,miter limit=3.25,line width=1.240pt] (275.7070,235.9805) -- (557.1641,440.4688);
  \path[cm={{1.0,0.0,0.0,1.0,(-66.0,-5.0)}},draw=black,line join=miter,line cap=round,miter limit=3.25,line width=1.240pt] (383.2148,201.0508) -- (275.7070,531.9219);
\end{scope}

%% file: hartshorne-sqrt2.tikz

\begin{scope}
  \path[cm={{1.0,0.0,0.0,1.0,(-66.0,-5.0)}},draw=black,line join=miter,line cap=round,miter limit=3.25,line width=1.240pt] (73.1797,356.5312) -- (417.8203,356.5312);
  \path[cm={{1.0,0.0,0.0,1.0,(-66.0,-5.0)}},draw=black,line join=miter,line cap=round,miter limit=3.25,line width=1.240pt] (73.1797,356.5312) -- (245.5000,11.8945);
  \path[cm={{1.0,0.0,0.0,1.0,(-66.0,-5.0)}},draw=black,line join=miter,line cap=round,miter limit=3.25,line width=1.240pt] (417.8203,356.5312) -- (245.5000,11.8945);
  \path[cm={{1.0,0.0,0.0,1.0,(-66.0,-5.0)}},draw=black,line join=miter,line cap=round,miter limit=3.25,line width=1.240pt] (159.3398,184.2109) -- (331.6602,184.2109);
  \path[cm={{1.0,0.0,0.0,1.0,(-66.0,-5.0)}},draw=black,line join=miter,line cap=round,miter limit=3.25,line width=1.240pt] (159.3398,184.2109) -- (245.5000,356.5312);
  \path[cm={{1.0,0.0,0.0,1.0,(-66.0,-5.0)}},draw=black,line join=miter,line cap=round,miter limit=3.25,line width=1.240pt] (245.5000,356.5312) -- (331.6602,184.2109);
  \path[cm={{1.0,0.0,0.0,1.0,(-66.0,-5.0)}},draw=black,line join=miter,line cap=round,miter limit=3.25,line width=1.240pt] (331.6602,184.2109) -- (123.6523,255.5898);
  \path[cm={{1.0,0.0,0.0,1.0,(-66.0,-5.0)}},draw=black,line join=miter,line cap=round,miter limit=3.25,line width=1.240pt] (123.6523,255.5898) -- (367.3477,255.5898);
  \path[cm={{1.0,0.0,0.0,1.0,(-66.0,-5.0)}},draw=black,line join=miter,line cap=round,miter limit=3.25,line width=1.240pt] (367.3477,255.5898) -- (73.1797,356.5312);
\end{scope}

%% file: staudt-addition.tikz

\begin{scope}
  \path[cm={{1.0,0.0,0.0,1.0,(-66.0,-5.0)}},draw=black,line join=miter,line cap=round,miter limit=3.25,line width=1.843pt] (66.0000,245.0000) -- (642.0000,245.0000) node[inner sep=5pt,fill=white] {\LARGE$\ell_x$};
  \path[cm={{1.0,0.0,0.0,1.0,(-66.0,-5.0)}},draw=black,line join=miter,line cap=round,miter limit=3.25,line width=1.843pt] (162.0000,341.0000) -- (162.0000,5.0000)  node[inner sep=5pt,fill=white] {\LARGE$\ell_y$};
  \path[cm={{1.0,0.0,0.0,1.0,(-66.0,-5.0)}},draw=black,dash pattern=on 18.43pt off 9.22pt,line join=miter,line cap=round,miter limit=3.25,line width=1.843pt] (66.0000,29.0000) -- (642.0000,317.0000) node [inner sep=5pt,fill=white] {\LARGE$j'$};
  \path[cm={{1.0,0.0,0.0,1.0,(-66.0,-5.0)}},draw=black,dash pattern=on 0.80pt off 3.20pt,line join=miter,line cap=round,miter limit=3.25,line width=1.843pt] (306.0000,341.0000) -- (306.0000,5.0000) node [inner sep=5pt,fill=white] {\LARGE$h$};
  \path[cm={{1.0,0.0,0.0,1.0,(-66.0,-5.0)}},draw=black,dash pattern=on 0.80pt off 3.20pt,line join=miter,line cap=round,miter limit=3.25,line width=1.843pt] (66.0000,149.0000) -- (642.0000,149.0000) node [inner sep=5pt,fill=white] {\LARGE$g$};
  \path[cm={{1.0,0.0,0.0,1.0,(-66.0,-5.0)}},draw=black,dash pattern=on 0.80pt off 3.20pt,line join=miter,line cap=round,miter limit=3.25,line width=1.843pt] (66.0000,101.0000) -- (546.0000,341.0000) node [inner sep=5pt,fill=white] {\LARGE$j$};
  \path[fill=black,even odd rule] (101.9805,240.0000) .. controls (101.9805,238.4141) and (101.3516,236.8906) .. (100.2305,235.7695) .. controls (99.1094,234.6484) and (97.5859,234.0195) .. (96.0000,234.0195) .. controls (94.4141,234.0195) and (92.8906,234.6484) .. (91.7695,235.7695) .. controls (90.6484,236.8906) and (90.0195,238.4141) .. (90.0195,240.0000) .. controls (90.0195,241.5859) and (90.6484,243.1094) .. (91.7695,244.2305) .. controls (92.8906,245.3516) and (94.4141,245.9805) .. (96.0000,245.9805) .. controls (97.5859,245.9805) and (99.1094,245.3516) .. (100.2305,244.2305) .. controls (101.3516,243.1094) and (101.9805,241.5859) .. (101.9805,240.0000) -- cycle(101.9805,240.0000) node[anchor=north,xshift=-5pt,yshift=-10pt,inner sep=5pt,fill=white] {\LARGE$\bm 0$};
  \path[fill=black,even odd rule] (245.9805,240.0000) .. controls (245.9805,238.4141) and (245.3516,236.8906) .. (244.2305,235.7695) .. controls (243.1094,234.6484) and (241.5859,234.0195) .. (240.0000,234.0195) .. controls (238.4141,234.0195) and (236.8906,234.6484) .. (235.7695,235.7695) .. controls (234.6484,236.8906) and (234.0195,238.4141) .. (234.0195,240.0000) .. controls (234.0195,241.5859) and (234.6484,243.1094) .. (235.7695,244.2305) .. controls (236.8906,245.3516) and (238.4141,245.9805) .. (240.0000,245.9805) .. controls (241.5859,245.9805) and (243.1094,245.3516) .. (244.2305,244.2305) .. controls (245.3516,243.1094) and (245.9805,241.5859) .. (245.9805,240.0000) -- cycle(245.9805,240.0000) node[anchor=north,xshift=-5pt,yshift=-13pt,inner sep=5pt,fill=white] {\LARGE$\bm x$};
  \path[fill=black,even odd rule] (293.9805,240.0000) .. controls (293.9805,238.4141) and (293.3516,236.8906) .. (292.2305,235.7695) .. controls (291.1094,234.6484) and (289.5859,234.0195) .. (288.0000,234.0195) .. controls (286.4141,234.0195) and (284.8906,234.6484) .. (283.7695,235.7695) .. controls (282.6484,236.8906) and (282.0195,238.4141) .. (282.0195,240.0000) .. controls (282.0195,241.5859) and (282.6484,243.1094) .. (283.7695,244.2305) .. controls (284.8906,245.3516) and (286.4141,245.9805) .. (288.0000,245.9805) .. controls (289.5859,245.9805) and (291.1094,245.3516) .. (292.2305,244.2305) .. controls (293.3516,243.1094) and (293.9805,241.5859) .. (293.9805,240.0000) -- cycle(293.9805,240.0000) node[anchor=north,xshift=-5pt,yshift=-13pt,inner sep=5pt,fill=white] {\LARGE$\bm y$};
  \path[fill=black,even odd rule] (101.9805,144.0000) .. controls (101.9805,142.4141) and (101.3516,140.8906) .. (100.2305,139.7695) .. controls (99.1094,138.6484) and (97.5859,138.0195) .. (96.0000,138.0195) .. controls (94.4141,138.0195) and (92.8906,138.6484) .. (91.7695,139.7695) .. controls (90.6484,140.8906) and (90.0195,142.4141) .. (90.0195,144.0000) .. controls (90.0195,145.5859) and (90.6484,147.1094) .. (91.7695,148.2305) .. controls (92.8906,149.3516) and (94.4141,149.9805) .. (96.0000,149.9805) .. controls (97.5859,149.9805) and (99.1094,149.3516) .. (100.2305,148.2305) .. controls (101.3516,147.1094) and (101.9805,145.5859) .. (101.9805,144.0000) -- cycle(101.9805,144.0000) node[anchor=east,xshift=-10pt,inner sep=5pt,fill=white] {\LARGE$\bm1_y$};
  \path[fill=black,even odd rule] (245.2188,144.0000) .. controls (245.2188,142.6172) and (244.6719,141.2891) .. (243.6914,140.3086) .. controls (242.7109,139.3281) and (241.3828,138.7812) .. (240.0000,138.7812) .. controls (238.6172,138.7812) and (237.2891,139.3281) .. (236.3086,140.3086) .. controls (235.3281,141.2891) and (234.7812,142.6172) .. (234.7812,144.0000) .. controls (234.7812,145.3828) and (235.3281,146.7109) .. (236.3086,147.6914) .. controls (237.2891,148.6719) and (238.6172,149.2188) .. (240.0000,149.2188) .. controls (241.3828,149.2188) and (242.7109,148.6719) .. (243.6914,147.6914) .. controls (244.6719,146.7109) and (245.2188,145.3828) .. (245.2188,144.0000) -- cycle(245.2188,144.0000) node[anchor=north,xshift=-5pt,yshift=-13pt,inner sep=5pt,fill=white] {\LARGE$\bm q$};
  \path[fill=white,even odd rule] (245.2188,144.0000) .. controls (245.2188,142.6172) and (244.6719,141.2891) .. (243.6914,140.3086) .. controls (242.7109,139.3281) and (241.3828,138.7812) .. (240.0000,138.7812) .. controls (238.6172,138.7812) and (237.2891,139.3281) .. (236.3086,140.3086) .. controls (235.3281,141.2891) and (234.7812,142.6172) .. (234.7812,144.0000) .. controls (234.7812,145.3828) and (235.3281,146.7109) .. (236.3086,147.6914) .. controls (237.2891,148.6719) and (238.6172,149.2188) .. (240.0000,149.2188) .. controls (241.3828,149.2188) and (242.7109,148.6719) .. (243.6914,147.6914) .. controls (244.6719,146.7109) and (245.2188,145.3828) .. (245.2188,144.0000) -- cycle(245.2188,144.0000);
    \path[cm={{1.0,0.0,0.0,1.0,(-66.0,-5.0)}},draw=black,line join=miter,line cap=rect,miter limit=3.25,line width=1.600pt] (311.2188,149.0000) .. controls (311.2188,147.6172) and (310.6719,146.2891) .. (309.6914,145.3086) .. controls (308.7109,144.3281) and (307.3828,143.7812) .. (306.0000,143.7812) .. controls (304.6172,143.7812) and (303.2891,144.3281) .. (302.3086,145.3086) .. controls (301.3281,146.2891) and (300.7812,147.6172) .. (300.7812,149.0000) .. controls (300.7812,150.3828) and (301.3281,151.7109) .. (302.3086,152.6914) .. controls (303.2891,153.6719) and (304.6172,154.2188) .. (306.0000,154.2188) .. controls (307.3828,154.2188) and (308.7109,153.6719) .. (309.6914,152.6914) .. controls (310.6719,151.7109) and (311.2188,150.3828) .. (311.2188,149.0000) -- cycle(311.2188,149.0000);
  \path[cm={{1.0,0.0,0.0,1.0,(-66.0,-5.0)}},draw=black,fill=black,dash pattern=on 0.80pt off 3.20pt,line join=miter,line cap=round,miter limit=3.24,nonzero rule,line width=0.024pt] (490.4023,252.5977) -- (490.4023,237.4023) -- (505.5977,237.4023) -- (505.5977,252.5977) -- cycle(490.4023,252.5977) node[anchor=north,xshift=6pt,yshift=-5pt,inner sep=5pt,fill=white] {\LARGE$\bm{x+y}$};
\end{scope}

%% file: staudt-multiplication.tikz

\begin{scope}
  \path[cm={{1.0,0.0,0.0,1.0,(-66.0,-5.0)}},draw=black,line join=miter,line cap=round,miter limit=3.25,line width=1.843pt] (66.0000,245.0000) -- (642.0000,245.0000) node[inner sep=5pt,fill=white] {\LARGE$\ell_x$};
  \path[cm={{1.0,0.0,0.0,1.0,(-66.0,-5.0)}},draw=black,line join=miter,line cap=round,miter limit=3.25,line width=1.843pt] (162.0000,341.0000) -- (162.0000,5.0000)  node[inner sep=5pt,fill=white] {\LARGE$\ell_y$};
  \path[cm={{1.0,0.0,0.0,1.0,(-66.0,-5.0)}},draw=black,dash pattern=on 18.43pt off 9.22pt,line join=miter,line cap=round,miter limit=3.25,line width=1.843pt] (450.0000,341.0000) -- (114.0000,5.0000) node [inner sep=5pt,fill=white,xshift=80pt,yshift=-105pt] {\LARGE$g'$};
  \path[cm={{1.0,0.0,0.0,1.0,(-66.0,-5.0)}},draw=black,dash pattern=on 18.43pt off 9.22pt,line join=miter,line cap=round,miter limit=3.25,line width=1.843pt] (594.0000,341.0000) -- (90.0000,5.0000) node [inner sep=5pt,fill=white,xshift=145pt,yshift=-75pt] {\LARGE$h'$};
  \path[cm={{1.0,0.0,0.0,1.0,(-66.0,-5.0)}},draw=black,dash pattern=on 0.80pt off 3.20pt,line join=miter,line cap=round,miter limit=3.25,line width=1.843pt] (66.0000,53.0000) -- (354.0000,341.0000) node [inner sep=5pt,fill=white,xshift=-120pt,yshift=100pt] {\LARGE$g$};
  \path[cm={{1.0,0.0,0.0,1.0,(-66.0,-5.0)}},draw=black,dash pattern=on 0.80pt off 3.20pt,line join=miter,line cap=round,miter limit=3.25,line width=1.843pt] (66.0000,85.0000) -- (450.0000,341.0000) node [inner sep=5pt,fill=white,xshift=-160pt,yshift=125pt] {\LARGE$h$};

  \path[fill=black,even odd rule] (101.9805,240.0000) .. controls (101.9805,238.4141) and (101.3516,236.8906) .. (100.2305,235.7695) .. controls (99.1094,234.6484) and (97.5859,234.0195) .. (96.0000,234.0195) .. controls (94.4141,234.0195) and (92.8906,234.6484) .. (91.7695,235.7695) .. controls (90.6484,236.8906) and (90.0195,238.4141) .. (90.0195,240.0000) .. controls (90.0195,241.5859) and (90.6484,243.1094) .. (91.7695,244.2305) .. controls (92.8906,245.3516) and (94.4141,245.9805) .. (96.0000,245.9805) .. controls (97.5859,245.9805) and (99.1094,245.3516) .. (100.2305,244.2305) .. controls (101.3516,243.1094) and (101.9805,241.5859) .. (101.9805,240.0000) -- cycle(101.9805,240.0000) node[anchor=north,xshift=-5pt,yshift=-10pt,inner sep=5pt,fill=white] {\LARGE$\bm 0$};

  \path[fill=black,even odd rule] (197.9805,240.0000) .. controls (197.9805,238.4141) and (197.3516,236.8906) .. (196.2305,235.7695) .. controls (195.1094,234.6484) and (193.5859,234.0195) .. (192.0000,234.0195) .. controls (190.4141,234.0195) and (188.8906,234.6484) .. (187.7695,235.7695) .. controls (186.6484,236.8906) and (186.0195,238.4141) .. (186.0195,240.0000) .. controls (186.0195,241.5859) and (186.6484,243.1094) .. (187.7695,244.2305) .. controls (188.8906,245.3516) and (190.4141,245.9805) .. (192.0000,245.9805) .. controls (193.5859,245.9805) and (195.1094,245.3516) .. (196.2305,244.2305) .. controls (197.3516,243.1094) and (197.9805,241.5859) .. (197.9805,240.0000) -- cycle(197.9805,240.0000) node[anchor=north,xshift=-5pt,yshift=-10pt,inner sep=5pt,fill=white] {\LARGE$\bm 1_x$};
  \path[fill=black,even odd rule] (245.9805,240.0000) .. controls (245.9805,238.4141) and (245.3516,236.8906) .. (244.2305,235.7695) .. controls (243.1094,234.6484) and (241.5859,234.0195) .. (240.0000,234.0195) .. controls (238.4141,234.0195) and (236.8906,234.6484) .. (235.7695,235.7695) .. controls (234.6484,236.8906) and (234.0195,238.4141) .. (234.0195,240.0000) .. controls (234.0195,241.5859) and (234.6484,243.1094) .. (235.7695,244.2305) .. controls (236.8906,245.3516) and (238.4141,245.9805) .. (240.0000,245.9805) .. controls (241.5859,245.9805) and (243.1094,245.3516) .. (244.2305,244.2305) .. controls (245.3516,243.1094) and (245.9805,241.5859) .. (245.9805,240.0000) -- cycle(245.9805,240.0000) node[anchor=north,xshift=-5pt,yshift=-13pt,inner sep=5pt,fill=white] {\LARGE$\bm x$};
  \path[fill=black,even odd rule] (293.9805,240.0000) .. controls (293.9805,238.4141) and (293.3516,236.8906) .. (292.2305,235.7695) .. controls (291.1094,234.6484) and (289.5859,234.0195) .. (288.0000,234.0195) .. controls (286.4141,234.0195) and (284.8906,234.6484) .. (283.7695,235.7695) .. controls (282.6484,236.8906) and (282.0195,238.4141) .. (282.0195,240.0000) .. controls (282.0195,241.5859) and (282.6484,243.1094) .. (283.7695,244.2305) .. controls (284.8906,245.3516) and (286.4141,245.9805) .. (288.0000,245.9805) .. controls (289.5859,245.9805) and (291.1094,245.3516) .. (292.2305,244.2305) .. controls (293.3516,243.1094) and (293.9805,241.5859) .. (293.9805,240.0000) -- cycle(293.9805,240.0000) node[anchor=north,xshift=-5pt,yshift=-13pt,inner sep=5pt,fill=white] {\LARGE$\bm y$};
  \path[fill=black,even odd rule] (101.9805,144.0000) .. controls (101.9805,142.4141) and (101.3516,140.8906) .. (100.2305,139.7695) .. controls (99.1094,138.6484) and (97.5859,138.0195) .. (96.0000,138.0195) .. controls (94.4141,138.0195) and (92.8906,138.6484) .. (91.7695,139.7695) .. controls (90.6484,140.8906) and (90.0195,142.4141) .. (90.0195,144.0000) .. controls (90.0195,145.5859) and (90.6484,147.1094) .. (91.7695,148.2305) .. controls (92.8906,149.3516) and (94.4141,149.9805) .. (96.0000,149.9805) .. controls (97.5859,149.9805) and (99.1094,149.3516) .. (100.2305,148.2305) .. controls (101.3516,147.1094) and (101.9805,145.5859) .. (101.9805,144.0000) -- cycle(101.9805,144.0000) node[anchor=east,xshift=-10pt,inner sep=5pt,fill=white] {\LARGE$\bm1_y$};
  \path[fill=black,even odd rule] (101.2188,48.0000) .. controls (101.2188,46.6172) and (100.6719,45.2891) .. (99.6914,44.3086) .. controls (98.7109,43.3281) and (97.3828,42.7812) .. (96.0000,42.7812) .. controls (94.6172,42.7812) and (93.2891,43.3281) .. (92.3086,44.3086) .. controls (91.3281,45.2891) and (90.7812,46.6172) .. (90.7812,48.0000) .. controls (90.7812,49.3828) and (91.3281,50.7109) .. (92.3086,51.6914) .. controls (93.2891,52.6719) and (94.6172,53.2188) .. (96.0000,53.2188) .. controls (97.3828,53.2188) and (98.7109,52.6719) .. (99.6914,51.6914) .. controls (100.6719,50.7109) and (101.2188,49.3828) .. (101.2188,48.0000) -- cycle(101.2188,48.0000) node [anchor=east,xshift=-10pt,inner sep=6pt,fill=white] {\LARGE$\bm q$};
  \path[fill=white,even odd rule] (101.2188,48.0000) .. controls (101.2188,46.6172) and (100.6719,45.2891) .. (99.6914,44.3086) .. controls (98.7109,43.3281) and (97.3828,42.7812) .. (96.0000,42.7812) .. controls (94.6172,42.7812) and (93.2891,43.3281) .. (92.3086,44.3086) .. controls (91.3281,45.2891) and (90.7812,46.6172) .. (90.7812,48.0000) .. controls (90.7812,49.3828) and (91.3281,50.7109) .. (92.3086,51.6914) .. controls (93.2891,52.6719) and (94.6172,53.2188) .. (96.0000,53.2188) .. controls (97.3828,53.2188) and (98.7109,52.6719) .. (99.6914,51.6914) .. controls (100.6719,50.7109) and (101.2188,49.3828) .. (101.2188,48.0000) -- cycle(101.2188,48.0000);
    \path[cm={{1.0,0.0,0.0,1.0,(-66.0,-5.0)}},draw=black,line join=miter,line cap=rect,miter limit=3.25,line width=1.600pt] (167.2188,53.0000) .. controls (167.2188,51.6172) and (166.6719,50.2891) .. (165.6914,49.3086) .. controls (164.7109,48.3281) and (163.3828,47.7812) .. (162.0000,47.7812) .. controls (160.6172,47.7812) and (159.2891,48.3281) .. (158.3086,49.3086) .. controls (157.3281,50.2891) and (156.7812,51.6172) .. (156.7812,53.0000) .. controls (156.7812,54.3828) and (157.3281,55.7109) .. (158.3086,56.6914) .. controls (159.2891,57.6719) and (160.6172,58.2188) .. (162.0000,58.2188) .. controls (163.3828,58.2188) and (164.7109,57.6719) .. (165.6914,56.6914) .. controls (166.6719,55.7109) and (167.2188,54.3828) .. (167.2188,53.0000) -- cycle(167.2188,53.0000);
  \path[cm={{1.0,0.0,0.0,1.0,(-66.0,-5.0)}},draw=black,fill=black,dash pattern=on 0.80pt off 3.20pt,line join=miter,line cap=round,miter limit=3.24,nonzero rule,line width=0.024pt] (442.4023,252.5977) -- (442.4023,237.4023) -- (457.5977,237.4023) -- (457.5977,252.5977) -- cycle(442.4023,252.5977) node[anchor=north,xshift=6pt,yshift=-5.5pt,inner sep=7pt,fill=white] {\LARGE$\bm{x \cdot y}$};
\end{scope}